%% file: gp-uq.tex
\documentclass[]{mpi2015-cscpreprint}
\usepackage{amssymb,amsmath,amsthm,mathrsfs,mathtools,amscd}
\usepackage{framed}

\usepackage{pgfplots}
\pgfplotsset{compat=newest}

\pgfplotsset{hide scale/.style={
/pgfplots/xtick scale label code/.code={},
/pgfplots/ytick scale label code/.code={}}}
\newlength\figureheight
\newlength\figurewidth

\input{def} 

\newtheorem{theorem}{Theorem}[section]
\newtheorem{cor}[theorem]{Corollary}

\begin{document}

\title{Space and Chaos-Expansion Galerkin POD Low-order Discretization of PDEs for Uncertainty Quantification} 
\shorttitle{Space-PCE POD for PDEs with Uncertainties}

\author[$\ast$]{Peter Benner}
\affil[$\ast$]{Max Planck Institute for Dynamics of Complex Technical Systems,
  Magdeburg, Germany.\authorcr
  Email: \texttt{\href{mailto:benner@mpi-magdeburg.mpg.de}{benner@mpi-magdeburg.mpg.de}},
  ORCID: \texttt{\href{https://orcid.org/0000-0003-3362-4103}%
    {0000-0003-3362-4103}}}

\author[$\ast\ast$]{Jan Heiland}
\affil[$\ast\ast$]{Max Planck Institute for Dynamics of Complex Technical Systems,
  Magdeburg, Germany. \newline
Faculty of Mathematics, Otto von Guericke University Magdeburg, Germany.
  \authorcr
  Email: \texttt{\href{mailto:heiland@mpi-magdeburg.mpg.de}{heiland@mpi-magdeburg.mpg.de}},
  ORCID: \texttt{\href{https://orcid.org/0000-0003-0228-8522}%
    {0000-0003-0228-8522}}}

\shortauthor{P. Benner, J. Heiland}

\keywords{uncertainty quantification, model reduction, Proper Orthogonal
Decomposition, tensor spaces}

\msc{35R60, 60H35, 65N22}

\abstract{
    The quantification of multivariate uncertainties in partial differential
    equations can easily exceed any computing capacity unless proper measures
    are taken to reduce the complexity of the model. In this work, we propose a
    multidimensional Galerkin Proper Orthogonal Decomposition that optimally
    reduces each dimension of a tensorized product space. We provide the
    analytical framework and results that define and quantify the
    low-dimensional approximation. We illustrate its application for uncertainty
    modeling with Polynomial Chaos Expansions and show its efficiency in a
    numerical example.
}

\maketitle
\tableofcontents

\section{Introduction}
The statistically sound treatment of modeled uncertainties in simulations comes
with significant additional computational costs.
Since a deterministic model can already be arbitrarily complex, the computation of
statistics for general problems may soon become infeasible unless some kind of
model reduction is involved.

In this work, we propose a multidimensional Galerkin POD that can simultaneously
and optimally reduces the physical dimensions of the model and the dimensions
related to the uncertainties.

For the quantification of uncertainties in PDE models and their numerical
discretization, one may distinguish two categories of solvers \cite{Soi17} --
sampling based methods, notably the \emph{Monte-Carlo method}, and Galerkin-type
projection methods. In this work, we focus on the latter. For a basic
explanation and relevant references on the \emph{Monte-Carlo method} and its
extensions see \cite{Soi17}, for an application in elliptic PDEs see
\cite{CliGST11}, and for a combination with \emph{stochastic collocation} and
tensor techniques see \cite{HajNTT16}.

Galerkin-type methods for solving PDEs with uncertainties are also referred to
as \emph{spectral stochastic methods} and base on a \emph{polynomial chaos
expansion} (PCE) of the candidate solution. If the involved random variable is univariate,
this means that the solution is formally expanded in a space of univariate
polynomials. These additional degrees of freedom then, via a Galerkin
projection with respect to a measure that encodes the statistical properties of
the involved uncertainty, fix the uncertainty in the solution. If the involved
randomness is multivariate, multivariate polynomials are used to resolve the
uncertainty. Since every dimension of the multivariation adds a dimension to the
problem, a numerical discretization quickly becomes infeasible in terms of
memory requirements, even if the dimensions are treated independent of each
other.

Several approaches to overcome this complexity have been proposed like
\emph{sparse grids} \cite{Gar13}, construction of reduced chaos expansions via,
say, \emph{Proper Generalized Decomposition} \cite{Nou10, TamMN14} or
\emph{Principal Component Analysis} or \emph{Karhunen-Lo\`eve} expansions
\cite{ArnGPR14,AudN12,BiaAAL15}, or the use of tensor formats to reduce
 or to handle the data more efficiently \cite{BalG15,BenOS15,Kho15,Ull10}.

The proposed approach develops a reduction method for tensorized PCE
approximations. For a given PCE, we define bases both for the spatial and
the uncertainty dimensions that optimally represent the data. These generated
low-dimensional bases drastically reduce the overall dimension and can be used
for efficient uncertainty quantification and, perspectively, for optimal control
of systems with uncertain parameters.

Finding optimal representations for the dimensions is comparable to identifying
low-rank tensor structures for the data, as it has been treated in
\cite{BalG15, BenOS15, GarU17, Kho15, KhoS11}. In contrast to these works, where
a predefined structure is adaptively filled to approximate the solution, we take
a given, possibly high-dimensional data set, and reduce it. The justification of
this top-down approach is that the obtained reduction is optimally fitted to the
given problem so that it can be used for further efficient explorations --
mainly because this approach admits a direct interpretation of the bases for
Galerkin discretizations. This relation to Galerkin projections defines the
common ground with the PGD approaches \cite{Nou10}, where optimal bases are
construction in an adaptive bottom-up fashion.

Most similar to our approach is the work \cite{ArnGPR14} on \emph{reduced chaos expansions} of
coupled systems, where, basically, a Galerkin POD approach is used for two uncertainty
dimensions. There, the authors start with a PCE of bivariate random coefficient
and obtain optimal bases via the left and right eigenvectors of a generalized
eigenvalue problem involving a covariance matrix and a mass matrix. This
approach via the eigenvectors of a covariance matrix is one way to define a POD
basis (see, e.g., \cite{Row05}) while the inclusion of the mass matrix provides
optimality in the relevant discrete function spaces; see \cite{BauHS15}. Our approach
extends the scope of this work by introducing the tensorized formulation that
allows for reduction of multivariate uncertainties together with the spatial
dimension in one framework.

The paper is organized as follows. In \Cref{sec:multi-dim-gal-pod}, we review
the space-time Galerkin POD approach and how it extends to problems with an
uncertainty dimension. Then we formulate the Galerkin POD for a product space of
arbitrary dimensions and provide the POD compression algorithms and results.
Next, in \Cref{sec:pce-product-space}, we show that a PCE discretization
exactly fits into this multidimensional Galerkin POD framework. In 
\Cref{sec:application-example}, we illustrate the use of the PCE and its POD
reduction for a generic linear convection diffusion PDE. Finally, in
\Cref{sec:num-example} we provide a numerical example that shows the
applicability and efficiency of this approach and show that a naive POD
reduction based on random snapshots is not useful for PDEs with uncertain
parameters.

\section{Multidimensional Galerkin POD}\label{sec:multi-dim-gal-pod}

In our previous work \cite{BauBH18}, we introduced space-time Galerkin POD. The idea of locating space and time dependent functions 
\begin{equation*}
    x \in L^2((0,T);L^2(\Omega)) \colon (0,T) \times \Omega \mapsto \mathbb R^{} ,
\end{equation*}
that, e.g., solve a partial differential equation, in the space-time product space 
\begin{equation*}
    L^2((0, T))\cdot L^2(\Omega)
\end{equation*}
naturally extends to functions that depend on space, time and a random parameter $\alpha$ 
\begin{equation*}
    x_\alpha \in L^2((0,T);L^2(\Omega)) \colon (0,T) \times \Omega \mapsto \mathbb R^{} 
\end{equation*}
in the space-time-uncertainty product space
\begin{equation*}
    L^2((0, T))\cdot L^2(\Omega)\cdot L^2(\Gamma, \mathbb P_\alpha),
\end{equation*}
where $\Gamma $ is the domain of the random parameter and $\mathbb P$ is the
associated probability measure; see, e.g., \cite{GarU17} where stationary
problems are treated in this setup.
   
Also, the approach of considering the approximation in the product of the
discrete spatial $\Ouss\subset L^2(\Omega)$ and time $\Oust\subset L^2((0,T))$ spaces extends to approximating $x_\alpha$ in  
\begin{equation*}
\Oustsr,
\end{equation*}
where \Ousr~is the finite dimensional space that models a polynomial chaos expansion of $L^2(\Gamma, \mathbb P_\alpha) $.

And, finally one may approximate a function $\mathbf x$ via its orthogonal projection onto $\hOustsr$, where 
\begin{equation*}
    \hOuss \subset \Ouss,\quad \hOust\subset \Oust, \quad\text{and}\quad \hOusr \subset \Ousr
\end{equation*}
were chosen optimally with respect to $\mathbf x$ for given dimensions of the subspaces. 

We provide a general formulation of the product spaces, their discretization, and their optimal low-dimensional approximation. For $i=1,2,\dotsc,N$, let
\begin{equation*}
    \Ousvi i := \spann\{\psi_i^1, \psi_i^2, \dotsc, \psi_i^{d_i}\}
\end{equation*}
be $d_i$ dimensional Hilbert spaces with inner product $\inpro{\cdot}{\cdot}_{\Ousvi i} $  and mass matrix 
\begin{equation*}
\Mvi i = 
\begin{bmatrix}
    \inpro{\psi_i^k}{\psi_i^\ell}_\Ousvi i
\end{bmatrix}_{i=1,\dotsc, d_i, \ell=1,\dotsc, d_i}
\in \mathbb{R}^{d_i,d_i}.
\end{equation*}
We will use the formal vector of the basis functions
\begin{equation}\label{eqn:formalvecbasfuns}
    \Psi _i = 
    \begin{bmatrix} 
        \psi_i^1 \\ \psi_i^2 \\ \vdots \\\psi_i^{d_i},
    \end{bmatrix}
\end{equation}
to write, e.g.,
\begin{equation*}
    \Mvi i = \inpro{\Psi_i}{\Psi_i^\trp}_{\Ousvi i},
\end{equation*}
via applying the functional $\inpro{\cdot}{\cdot}\colon \psi_i^\ell \psi_i^k \mapsto \inpro{\psi_i^\ell}{ \psi_i^k}$ pointwise to the entries of the formal matrix ${\Psi_i}{\Psi_i^\trp}$.
Finally, let $\Lvi i \in \mathbb{R}^{d_i,d_i}$  be a factor such that 
\begin{equation*}
    \Mvi i= \Lvi i\Lvit i.
\end{equation*}

We consider the product space
\begin{equation*}
    \mathcal V = \prod_{i=1}^N \Ousvi i
\end{equation*}
of spaces of square integrable functions with the inner product 
\begin{equation*}
\inprov yz = \inprovi{y_i}{z_i},
\end{equation*}
where $\inva{}_i$ denotes the measure associated with $\Ousvi i$.

We represent a function $x\in \mathcal V$ via
\begin{equation*}
    x = \sum_{k_1 = 1}^{d_1}\sum_{k_2 = 1}^{d_2} \dotsm \sum_{k_N = 1}^{d_N} \xkotkn \psi_1^{k_1}\psi_2^{k_2}\dotsm\psi_N^{k_N}
\end{equation*}
or, equivalently, via the $N$-dimensional tensor of the coefficients
\begin{equation*}
    \mathbf X = \bigl[ \xkotkn  \bigr].
\end{equation*}
Note that 
\begin{equation}\label{eqn:functionvstensor}
    x = \Vec(\tnsr X)^\trp \bigl [\Psi_N \otimes \dotsm \otimes \Psi_2 \otimes \Psi_1 \bigr].
\end{equation}
\begin{theorem}\label{thm:tensornormasfrobnorm}
    For a function $x\in \mathcal V$ with its representation \tnsr X~as in
    \eqref{eqn:functionvstensor}, one has
    \begin{align*}
        \|x\|^2_{\Ousv} &=\int\int \dotsm \int x^2 \inva{}_1\inva{}_2 \dotsm \inva{}_N \\
               &= \| \Lvit 1 \tnsr X^{(1)}\bigl [\Lvi N \otimes \dotsm \otimes
               \Lvi 2  \bigr] \|_F^2,
    \end{align*}
where $\tnsr X ^{(1)}$ is the mode-1 matricization of the coefficient tensor
$\tnsr{X}$.
\end{theorem}
\begin{proof}
    We use the properties of the \emph{Kronecker}-product $\otimes$, the
    $\mu$-mode tensor product $\circ_\mu$, the vectorization operator $\Vec$,
    and the $\mu$-mode matricization operator $\cdot ^{(\mu)}$ to directly
    compute
    \begin{align*}
        \|x\|^2_{\Ousv} &=\int\int \dotsm \int x^2 \inva{}_1\inva{}_2 \dotsm \inva{}_N \\
               &= \Vec(\tnsr X)^\trp\int\int \dotsm \int \bigl [\Psi_N\Psi_N^{\trp} \otimes \dotsm \otimes \Psi_2\Psi_2^{\trp} \otimes \Psi_1\Psi_1^{\trp} \bigr] \inva{}_1\inva{}_2 \dotsm \inva{}_N \Vec(\tnsr X)\\
               &= \Vec (\tnsr X)^\trp \bigl [\Mvi N \otimes \dotsm \otimes \Mvi 2 \otimes \Mvi 1 \bigr] \Vec(\tnsr X) \\
               &= \|\bigl [\Lvit N \otimes \dotsm \otimes \Lvit 2 \otimes \Lvit 1 \bigr] \Vec(\tnsr X) \|_2^2 \\
               &= \|\bigl [\Lvit N \otimes \dotsm \otimes \Lvit 2 \otimes I \bigr]\bigl [I \otimes \dotsm \otimes I \otimes \Lvit 1 \bigr] \Vec(\tnsr X) \|_2^2 \\
               &= \|\bigl [\Lvit N \otimes \dotsm \otimes \Lvit 2 \otimes I \bigr]\Vec(\Lvit 1 \circ_1 \tnsr X) \|_2^2 \\
               &= \|\Vec\bigl ( \bigl [\Lvit N \otimes \dotsm \otimes \Lvit 2 \bigr]\circ_2 (\Lvit 1 \circ_1 \tnsr X)\bigr ) \|_2^2 \\
               &= \| \bigl [\Lvit N \otimes \dotsm \otimes \Lvit 2 \bigr]\circ_2 (\Lvit 1 \circ_1 \tnsr X) \|_F^2 \\
               &= \| (\Lvit 1\circ_1 \tnsr X)^{(1)}\bigl [\Lvi N \otimes \dotsm \otimes \Lvi 2  \bigr] \|_F^2 \\
               &= \| \Lvit 1 \tnsr X^{(1)}\bigl [\Lvi N \otimes \dotsm \otimes \Lvi 2  \bigr] \|_F^2.
    \end{align*}
\end{proof}

    By permutations of the tensor $\tnsr X$, the dimension associated with any
    $\Ousvi i$ can take the role of the first dimension with $\Lvi 1$ in the
    formula of \Cref{thm:tensornormasfrobnorm}. To avoid technicalities,
    we will consider permutations that simply cycle through the dimensions.
    Therefore, we introduce the operator that permutes a tensor
    \begin{equation*}
        \pt
        \colon \tnsr X\in\mathbb R^{d_1,d_2, \dotsc, d_N}   \mapsto \pt{\tnsr X} \in \mathbb R^{d_2, \dotsc, d_N,d_1}
    \end{equation*}
    via 
    \begin{equation*}
    \bigl[ (\pt \mathbf x)^{k_1 k_2 \dotsm k_N} \bigr]
    =
    \bigl[ \mathbf x^{k_2 \dotsm k_N k_1 } \bigr].
\end{equation*}
Note that $\pt^N \tnsr X = \tnsr X$ and that, for matrices $\mathbf M$ (where
$N=2$), it holds that $\pt \tnsr M = \tnsr M^\trp$.

\begin{cor}[of \Cref{thm:tensornormasfrobnorm}]\label{cor:normwcycleddims}
    For any $i\in\{1,\dotsc, N\}$, the norm of $x\in\Ousv$ can be expressed as
    \begin{equation*}
        \|x\|_{\Ousv}^2 = \| \Lvit i (\pt^{i-1} \tnsr X)^{(1)}\bigl [\Lvi {i-1} \otimes \dotsm \otimes \Lvi 1 \otimes \Lvi N \otimes \dotsm \otimes \Lvi {i+1}  \bigr] \|_F^2,
    \end{equation*}
    with the convention that $\Lvi {i-1} \otimes \dotsm \otimes \Lvi 1$ is void for $i=1$ as is $\Lvi N \otimes \dotsm \otimes \Lvi {i+1}$ for $i=N$.
\end{cor}
With these expressions for the norm of the function $x\in\Ousv$ related to a tensor $\tnsr X$ via \eqref{eqn:functionvstensor}, we can provide an interpretation of the \emph{higher-order singular value decomposition} \cite{DeDV00} in terms of low-dimensional space discretizations as it is the backbone of the POD.
\begin{theorem}\label{thm:multidimpodspaces}
    Given $x\in\Ousv$. For any $i\in\{1,\dotsc, N\} $ and for a corresponding $\hat d_i \leq d_i$, the space spanned by 
    \begin{equation*}
        \hat \Psi_i  =
    \begin{bmatrix} 
        \hpsi i1 \\ \hpsi i2 \\ \vdots \\\hpsi i{\hat d_i}
    \end{bmatrix}
    :=V_{i,\hat d_i}^\trp  \Lvi i^{-1} 
    \begin{bmatrix} 
        \ppsi i1 \\ \ppsi i2 \\ \vdots \\\ppsi i{d_i}
    \end{bmatrix}
    =V_{i,\hat d_i}^\trp  \Lvi i^{-1} \Psi_i,
    \end{equation*}
where $V_{i,\hat d_i}$ is the matrix of the $\hat d_i$ leading left singular vectors of 
\begin{equation*}
\Lvit i (\pt^{i-1} \tnsr X)^{(1)}\bigl [\Lvi {i-1} \otimes \dotsm \otimes \Lvi 1 \otimes \Lvi N \otimes \dotsm \otimes \Lvi {i+1}  \bigr],
\end{equation*}
optimally approximates $\Ousvi i$ in the sense that $x$ is best approximated in 
\begin{equation*}
    \Ousvi 1\cdot \Ousvi 2 \dotsm \Ousvi {i-1} \cdot \hat {\Ousv}_i \cdot \Ousvi {i+1} \dotsm \Ousvi N
\end{equation*}
in the \Ousv-norm over all subspaces of $\Ousvi i$ of dimension $\hat d_i$.
\end{theorem}

\begin{proof}
    For $i=1$, the claim follows directly from \cite[Lem. 2.5]{BauBH18} with considering $\Ousvi 1 \cdot \mathcal W$, and $\mathcal W:= \Ousvi 2\cdot \Ousvi 2 \dotsm \Ousvi N$. For any other $i$, one can apply Corollary \ref{cor:normwcycleddims} first.
\end{proof}

For the overall projection error between $x$ and its projection $\hat x$ onto 
\begin{equation*}
    \hat {\Ousv}_1 \cdot \hat {\Ousv}_2 \cdot \dotsm \cdot \hat {\Ousv}_N
\end{equation*}
with $\hat{\Ousv}_i$ of dimension $\hat d_i$ as defined in
\Cref{thm:multidimpodspaces}, one has that
\begin{equation}\label{eqn:projectionerror}
    \|x-\hat x\|_{\Ousv}^2 \leq \sum_{k_1=\hat d_1 +1}^{d_1} {\sigma_{k_1}^{(1)}}^2
    + \sum_{k_2=\hat d_2 +1}^{d_2} {\sigma_{k_2}^{(2)}}^2
    + \dotsm
    + \sum_{k_N=\hat d_N +1}^{d_N} {\sigma_{k_N}^{(N)}}^2,
\end{equation}
where $\sigma_{k}^{(i)}$ is the $k$-th singular value of $\tnsr X^{(i)}$ as they
appear in the SVD for the definition of $\hat{\Ousv}_i$. The estimate
\eqref{eqn:projectionerror} follows directly from \cite[Eqn. (24)]{DeDV00} if
one takes into account the scalings by the factors of the mass matrices. Note
that while a single $\Ousvi i$ is optimally approximated by $\hat \Ousv _i $ by
virtue of \Cref{thm:multidimpodspaces}, the approximation of $\Ousv$ by
$\prod _{i=1}^N \hat{\Ousv}_i$ might not be optimal in the same sense; see the
discussion in \cite[p. 1267]{DeDV00}.

\section{Polynomial Chaos Expansion as Product Space}\label{sec:pce-product-space}

Let
\begin{equation*}
    \alpha  = \begin{pmatrix} \alpha_1 , \alpha_2 , \dotsc, \alpha_N \end{pmatrix}
\end{equation*}
be a tuple of random variables $\alpha_i$ that take on values in a domain $\Gamma_i\subset \mathbb R$ and that are distributed according to a probability measure $\pinva{\alpha_i}$. If $\tilde y$ is a function that depends on $\alpha$, that for every realization of $\alpha$ takes on values in a Hilbert space, say, $L^2(\Omega)$ for a domain $\Omega$ in $\mathbb R^{2}$ or $\mathbb R^{3}$, and that has a bounded variance with respect to $\alpha$, one may approximate $\tilde y$ by a suitable
\begin{equation}\label{eqn:fempceansatz}
    y \in L^2(\Omega) \cdot \Ltgi 1 \cdot \Ltgi 2 \cdot \dotsm \cdot \Ltgi N.
\end{equation}
Note that $y$ is a random variable and that the expected value $\mathbb E y \in \Lto $  of $y$ is defined as
\begin{equation*}
    \mathbb E y = \expova{y}.
\end{equation*}
A finite dimensional approximation $y$ to $\tilde y$ can be sought in 
\begin{equation}\label{eqn:fempceproductansatz}
    \Ousv = \Ousvi 0 \cdot \Ousvi 1 \cdot \Ousvi 2 \cdot \dotsm \cdot \Ousvi N
\end{equation}
where $\Ousvi 0 \subset \Lto$ is a Finite Element space and where, for $i=1, \dotsc, N$, $\Ousvi i$ is a finite dimensional subspace of $\Ltgi i$ derived from a \emph{Polynomial Chaos Expansion}. Here we will consider $d_i$-dimensional spaces
\begin{equation*}
    \Ousvi i = \spann\{\psij i1, \psij i2, \dotsc, \psij i{d_i}\},
\end{equation*}
with $\psij ik$ being the Lagrange polynomials of degree $d_i-1$ defined through
the distinct nodes
\begin{equation*}
\{\alphaij i1, \alphaij i2, \dotsc, \alphaij i{d_i} \} \subset \Gamma_i.
\end{equation*}
As for the nodes, we choose the Gaussian quadrature nodes
with respect to the measure $\nspinva i$; see \cite{FerA07} for formulas and
algorithms. With the corresponding quadrature weights
\begin{equation*}
    \{\wij i1, \wij i2, \dotsc, \wij i{d_i} \} \subset \mathbb R^{d_i},
\end{equation*}
the quadrature formula
\begin{equation}\label{eqn:probquadrt}
    \intgi i{z(\alpha)} \approx \sum_{k=1}^{d_i} \wij ik z(\alphaij ik)
\end{equation}
is exact for polynomials up to degree $2d_i -1$. By virtue of this exactness,
and since the Lagrange polynomials are orthogonal and fulfill $\psij ik(\alphaij ij)=1$ if $k=j$ and $\psij ik(\alphaij ij)=0$
if $k\neq j$, for the mass matrix $\Mvi i$, one has that 
\begin{equation*}
    \Mvi i = \intgi i{\Psi_i \Psi_i^\trp} = 
    \begin{bmatrix}
        \wij i1 &&& \\
        & \wij i2 && \\
        && \ddots & \\
        &&& \wij i{d_i} 
    \end{bmatrix}.
\end{equation*}

\section{Application Example}\label{sec:application-example}
For a domain $\Omega\subset\mathbb R^{d}$, with $d=2$ or $d=3$, for a
given-right hand side $f\in L^2(\Omega)$ and a given vector field $b\in [L^2(\Omega)]^d$,
we consider the generic \emph{convection-diffusion} problem
\begin{equation}\label{eqn:poissonproblem}
    b\cdot \nabla y- \nabla\cdot ( \kappa_\alpha \nabla y) = f,
\end{equation}
where we assume that the diffusivity coefficient depends on a random vector $\alpha=(\alpha_1, \dotsc , \alpha_N)$.

For the derivation, we assume homogeneous Dirichlet conditions or homogeneous
Neumann conditions for the boundary. Nonzero boundary conditions can be included
in standard ways.

If, for given $f$ and $b$, system \eqref{eqn:poissonproblem} has a solution $y$ for any realization of $\alpha$, then $y$ itself can be seen as a random variable depending on $\alpha$. 

As in standard finite element approaches, for every realization $\alpha$, we locate the corresponding solution $y_\alpha$ in $H_0^1(\Omega)$ and require \eqref{eqn:poissonproblem} to hold in the weak sense, namely 
\begin{equation}\label{eqn:poissonweak}
    \int_\Omega v(x) b(x)\cdot \nabla y_\alpha (x) + \kappa_\alpha \nabla v(x)\cdot \nabla y_\alpha (x) \inva x = \int_\Omega v(x)f(x) \inva x
\end{equation}
for all $v\in \Hoi$.

To account for the uncertainty, we assume the solution in the product space of
the space variable and the uncertainty dimensions as in
\eqref{eqn:fempceansatz} and require \eqref{eqn:poissonweak} to hold in expectation, i.e.

\begin{equation}\label{eqn:poissonuncrtnweak}
    \begin{split}
    \expova{\int_\Omega v b\cdot \nabla y + \kappa_\alpha \nabla v\cdot \nabla y \inva x }= \\
    \expova{\int_\Omega vf \inva x},
\end{split}
\end{equation}
where now $v$ is a trial function from the ansatz space 
\begin{equation*}
    \Hoi \cdot \Ltgi 1 \cdot \Ltgi 2 \cdot \dotsm \cdot \Ltgi N.
\end{equation*}

We may cluster the uncertainty dimensions $\Gamma_i$ into $\Gamma$ and write 
\begin{equation*}
    \int_\Gamma v(\alpha) \pinva{} \quad \text{instead of}\quad \expova{v(\alpha_1,
    \dotsc, \alpha_N)}.
\end{equation*}

For a finite dimensional approximation, let the FEM space $\Ousv_0$ be spanned
by $\Psi_0$ (compare \eqref{eqn:formalvecbasfuns}) and let $\Aal\in\mathbb
R^{d_0,d_0}$  be the discrete convection/diffusion operator:
\begin{equation*}
    \Aal = \int_\Omega \Psi_0 (b\cdot \nabla\Psi _0 ^\trp - \nabla\cdot \kappa_\alpha \nabla \Psi_0^\trp ) \inva x 
    = \int_\Omega \Psi_0 b\cdot \nabla\Psi _0 ^\trp + \kappa_\alpha \nabla \Psi_o\cdot \nabla \Psi_0 ^\trp \inva x,
\end{equation*}
where the products and the application of the differential operators are
understood component-wise. 

To save space in the formal derivation of the equations that include the
Polynomial Chaos Expansions we will \emph{formally} use the \emph{strong}
differential operator
\begin{equation*}
    a_\alpha\colon y \mapsto b\cdot \nabla y - \nabla \cdot (\kappa_\alpha \nabla y).
\end{equation*}
   
With that, with discrete ansatz spaces as in \eqref{eqn:fempceproductansatz}, and with the ansatz for the solution
\begin{equation}\label{eqn:ansatzforvecy}
    y = \Vec(\tnsr Y)^\trp \bigl [\Psi_N \otimes \dotsm \otimes \Psi_1  \otimes \Psi_0 \bigr ] =
    \bigl [\Psi_N^\trp \otimes \dotsm \otimes \Psi_1^\trp  \otimes \Psi_0^\trp \bigr ]\Vec(\tnsr Y),
\end{equation}
where $\tnsr Y$ is the tensor of coefficients (cp.
\eqref{eqn:functionvstensor}), Equation \eqref{eqn:poissonuncrtnweak} is
discretized as 
\begin{equation*}
    \int_\Gamma \int_\Omega 
    \bigl [\Psi_N \otimes \dotsm \otimes  \Psi_1  \otimes \Psi_0 \bigr ] a_\alpha y \inva x \pinva{} 
    = \int_\Gamma \int_\Omega \bigl [\Psi_N \otimes \dotsm \otimes \Psi_1  \otimes \Psi_0 \bigr ] f \inva x \pinva{},
\end{equation*}
where the left hand side, together with \eqref{eqn:ansatzforvecy}, becomes 
\begin{equation}\label{eqn:galerkinprojsys}
    \begin{split}
    \int_\Gamma \int_\Omega 
     \bigl [&\Psi_N \otimes \dotsm \otimes \Psi_1  \otimes \Psi_0 \bigr ] a_\alpha \bigl [\Psi_N^\trp \otimes \dotsm \otimes \Psi_1^\trp  \otimes \Psi_0^\trp \bigr ] \inva x \pinva{} \Vec(\tnsr Y) = \\
    &= \int_\Gamma \int_\Omega 
    \bigl [\Psi_N \otimes \dotsm \otimes \Psi_1  \otimes \Psi_0 \bigr ] \bigl [\Psi_N^\trp \otimes \dotsm \otimes \Psi_1^\trp  \otimes a_\alpha \Psi_0^\trp \bigr ] \inva x \pinva{} \Vec(\tnsr Y) \\
    &= 
    \int_\Gamma \int_\Omega 
    \bigl [\Psi_N\Psi_N^\trp  \otimes \dotsm \otimes \Psi_1 \Psi_1^\trp   \otimes \Psi_0 a_\alpha \Psi_0^\trp  \bigr ] \inva x \pinva{} \Vec(\tnsr Y) \\
    &= 
    \int_\Gamma 
    \bigl [\Psi_N\Psi_N^\trp  \otimes \dotsm \otimes \Psi_1 \Psi_1^\trp   \otimes \int_\Omega \Psi_0 a_\alpha \Psi_0^\trp  \inva x \bigr ] \pinva{} \Vec(\tnsr Y) \\
    &=
    \expova{
    \bigl [\Psi_N\Psi_N^\trp  \otimes \dotsm \otimes \Psi_1 \Psi_1^\trp   \otimes \Aal\bigr ] 
    }
    \Vec(\tnsr Y)
\end{split}
\end{equation}
thanks to the linearity of the involved differential operators and the Kronecker products. 

Next we successively approximate the integrals with respect to the probability measures by
the corresponding quadrature rules (cp. \eqref{eqn:probquadrt}) to obtain 
\begin{equation*}
    \begin{split}
    \expova{
    \bigl [\Psi_N\Psi_N^\trp  \otimes \dotsm \otimes \Psi_1 \Psi_1^\trp   \otimes \Aal\bigr ] 
    }
    \Vec(\tnsr Y) \\
    \approx
    \int_{\Gamma_N} \dotsm \int_{\Gamma_2} \sum_{k_1=1}^{d_1} \wij 1{k_1}
    \bigl [\Psi_N \Psi_N^\trp  \otimes \dotsm \otimes \Psi_1(\alphaij 1{k_1})
        \Psi_1(\alphaij 1{k_1})^\trp   \otimes
    A_{\alpha_1^{k_1}, \dotsc, \alpha_N} \bigr ] \pinva 2 \dotsm \pinva N\Vec(\tnsr Y).
\end{split}
\end{equation*}
Since the Lagrange polynomials are a nodal basis, it holds that for all $i=1,
\dotsc, N$, that $\Psi_i(\alphaij i{k_i})=e_{k_i}$, where $e_{k_i}\in \mathbb
R^{d_i}$ is the $k_i$-th canonical basis vector.  Accordingly, the coefficient
matrix for $\Vec(\tnsr Y)$ becomes 
\begin{equation*}
    \sum_{k_N=1}^{d_N}\dotsm \sum_{k_2=1}^{d_2}\sum_{k_1=1}^{d_1}\wij
    N{k_N}\dotsc \wij 2{k_2} \wij 1{k_1}
    \bigl [e_{k_N}e_{k_N}^\trp  \otimes \dotsm \otimes e_{k_2}e_{k_2}^\trp
        \otimes e_{k_1}e_{k_1}^\trp \otimes
        A_{\alpha_1^{k_1}, \dotsc, \alpha_N^{k_N}}
\bigr ],
\end{equation*}
which is a completely decoupled system for every combination $(\alpha_1^{k_1},
\dotsc, \alpha_N^{k_N})$.


To derive the Galerkin POD reduced system, we replace $\Psi_i$ by $\hat \Psi
_i$, for $i=0, 1, \dotsc , N$ in \eqref{eqn:galerkinprojsys}. We assume that the
reduced bases were obtained as proposed by \Cref{thm:multidimpodspaces}.
The derivation, however, works for any (reduced) basis.

For illustration, we consider the case $N=1$, i.e., the spatial dimension and a
univariate uncertainty. Then, the reduced system coefficient matrix reads
\begin{equation}\label{eqn:redgalerkinprojsys}
  \begin{split}
    \int_\Gamma \int_\Omega 
    \bigl [\hPsi_1 \hPsi_1^\trp
    \otimes \hPsi_0 a_\alpha \hPsi_0^\trp  \bigr ] & \inva x \pinva{} \Vec(\htnsr Y) 
    = \\
                                                   &=
    \int_\Gamma 
    \bigl [\hPsi_1 \hPsi_1^\trp
    \otimes \int_\Omega \hPsi_0 a_\alpha \hPsi_0^\trp  \inva x \bigr ]
    \pinva{} \Vec(\htnsr Y) \\
                                                   &= 
    \int_\Gamma{
      \bigl [\hPsi_1 \hPsi_1^\trp   \otimes \hAal\bigr ] 
    }\pinva{} 
    \Vec({\htnsr Y})\\
                                                   &= 
    \sum_{k_1=1}^{d_1} \wij 1{k_1}
    \bigl [\hPsi_1(\alphaij 1{k_1}) \hPsi_1(\alphaij 1{k_1})^\trp   \otimes
    \hat {\mathbf A}_{\alpha_1^{k_1}} \bigr ] \Vec(\htnsr Y).
  \end{split}
\end{equation}
Here, the operator $\hAal$ is the POD projection of $\Aal$:
\begin{equation*}
  \hAal = V_{0,\hat d_0}^\trp  \Lvi 0^{-1} \Aal  \Lvi 0^{-\trp} V_{0,\hat d_0},
\end{equation*}
whereas, for the given choice of $\Psi_1$ and the weights $\wij 1{k_1}$,
$k_1=1,\dotsc, d_1$, one has 
\begin{equation*}
  \begin{split}
    \wij 1{k_1} \bigl [\hPsi_1(\alphaij 1{k_1}) \hPsi_1(\alphaij 1{k_1})^\trp \bigr ]
    &= 
    \wij 1{k_1} V_{1,\hat d_1}^\trp  \Lvi 1^{-1}\bigl [\hPsi_1(\alphaij 1{k_1}) \hPsi_1(\alphaij 1{k_1})^\trp \bigr ]\Lvi 1^{-\trp} V_{1,\hat d_1}
    \\ &=
    \wij 1{k_1}V_{1,\hat d_1}^\trp \bigl(\wij 1{k_1}\bigr)^{-1/2}
    e_{k_1}e_{k_1}^\trp\bigl(\wij 1{k_1}\bigr)^{-1/2}V_{1,\hat
    d_1} \\
       &=
    V_{1,\hat d_1}^\trp e_{k_1}e_{k_1}^\trp V_{1,\hat d_1};
  \end{split}
\end{equation*}
cp. \Cref{thm:multidimpodspaces}. By orthonormality of the POD basis, one
obtains that 
\begin{equation*}
\sum_{k_1=1}^{d_1} \wij 1{k_1} \bigl [\hPsi_1(\alphaij 1{k_1})
\hPsi_1(\alphaij 1{k_1})^\trp\bigr] = I;
\end{equation*}
see \cite[Rem. 2.6]{BauBH18},
which, however, does not help when the terms are multiplied with the
(non-constant) $\Aal$ in \eqref{eqn:redgalerkinprojsys}. Accordingly, the reduced
system does not decouple and, in this univariate case, requires the solution of
a $\hat d_0 \cdot \hat d_1$-dimensional system.

The derivation of the general reduced multivariate systems goes along the same
lines and results in a possibly fully coupled system of dimension
$\prod_{j=0}^N\hat d_j$ which still can be prohibitively large. For these cases,
one may consider leaving a certain dimension, say $\Gamma_N$, unreduced and
rather solve $d_N$ systems of size $\prod_{j=0}^{N-1}\hat d_j$. Using this idea
recursively one can balance the number of systems and their size.

\section{Numerical Example}\label{sec:num-example}

\begin{figure}
    \begin{center}
    \includegraphics[width=.45\textwidth]{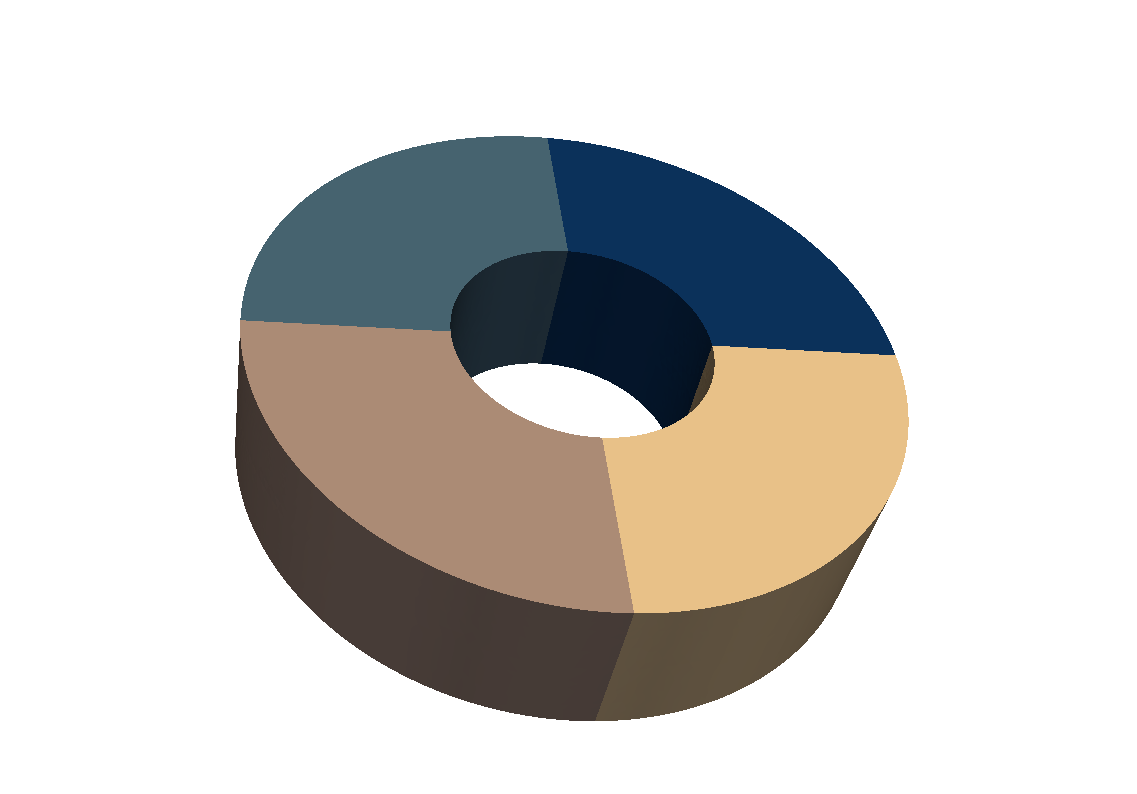}
    \includegraphics[width=.45\textwidth]{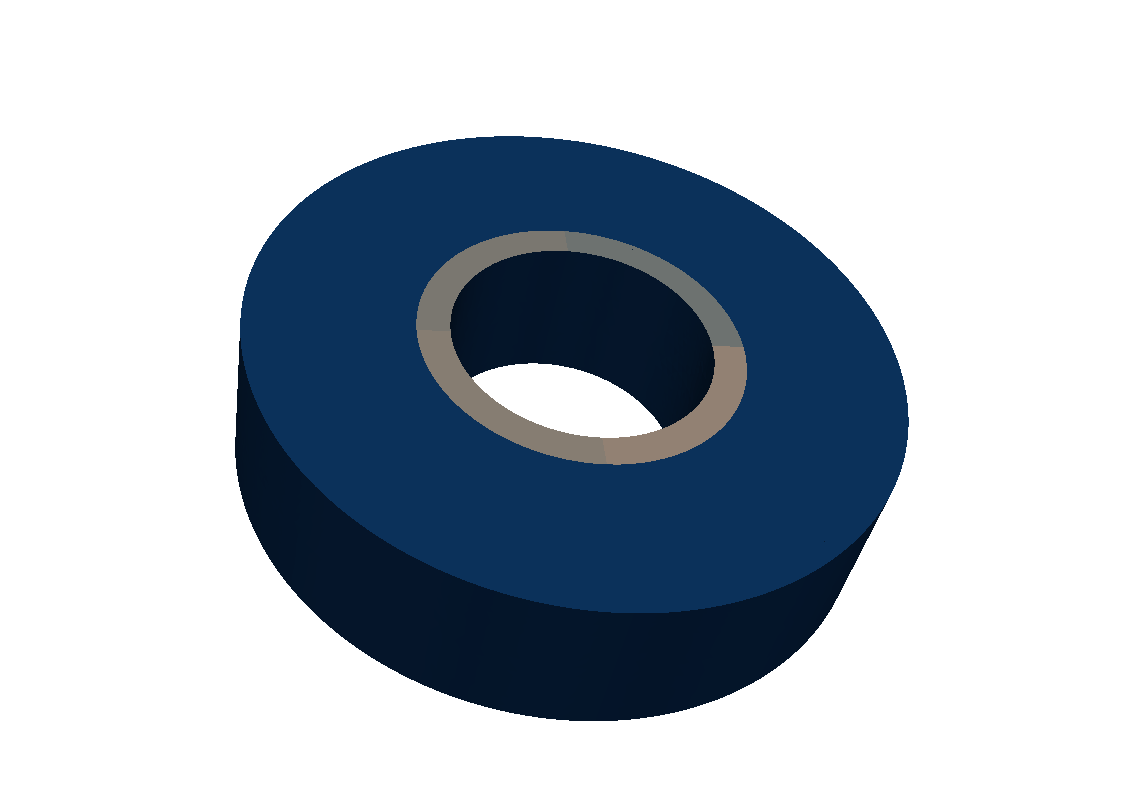}
\end{center}
\caption{The subdomains and the boundary patch for the measurements.}\label{fig:phys-facet-regs}
\end{figure}
\begin{figure}
\begin{minipage}{.6\textwidth}
    \includegraphics[width=\textwidth]{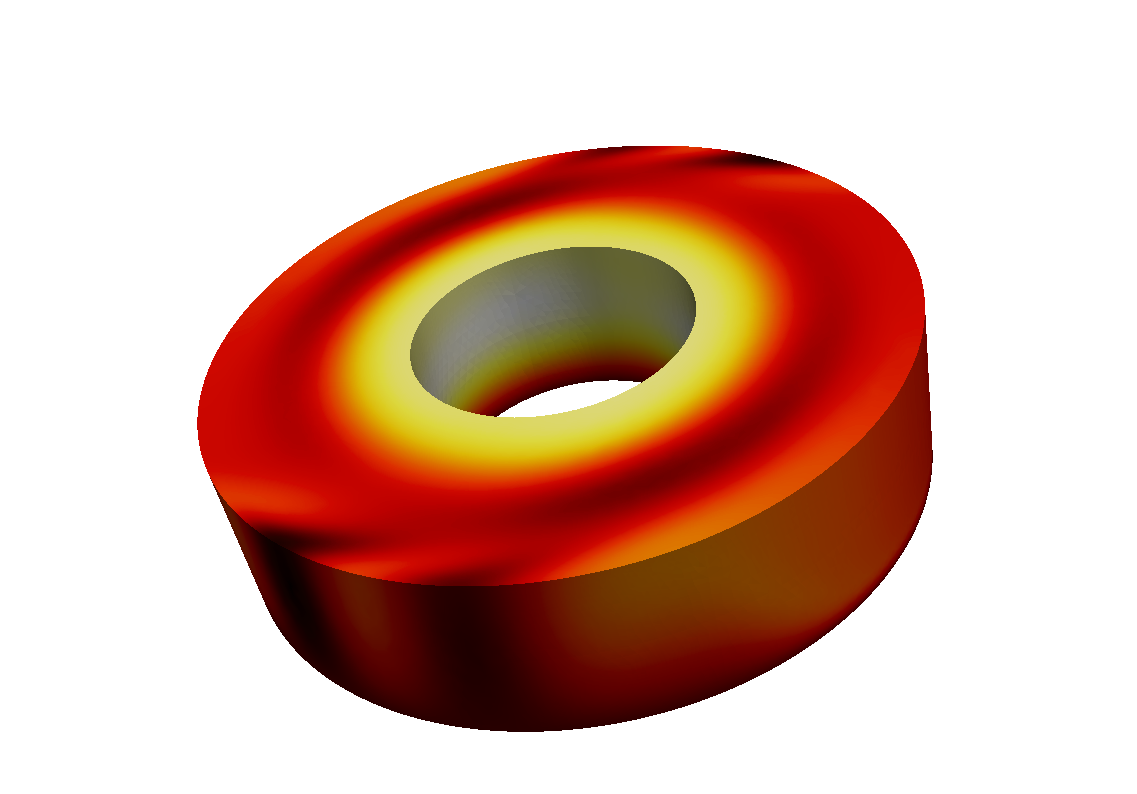}
  \end{minipage}
\begin{minipage}{.3\textwidth}
  \setlength\figureheight{5cm}
  \setlength\figurewidth{1.8cm}
  \includegraphics[]{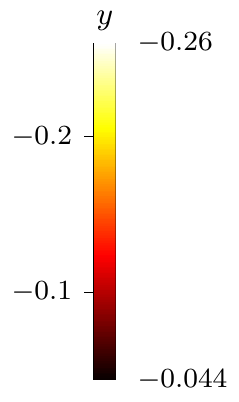}
    \end{minipage}
\caption{The solution $y$ for $\bar \alpha$ that is $\nu = 5\cdot 10^{-4}$.}
\label{fig:sol}
\end{figure}

\begin{figure}
\begin{minipage}{.6\textwidth}
      \includegraphics[width=\textwidth]{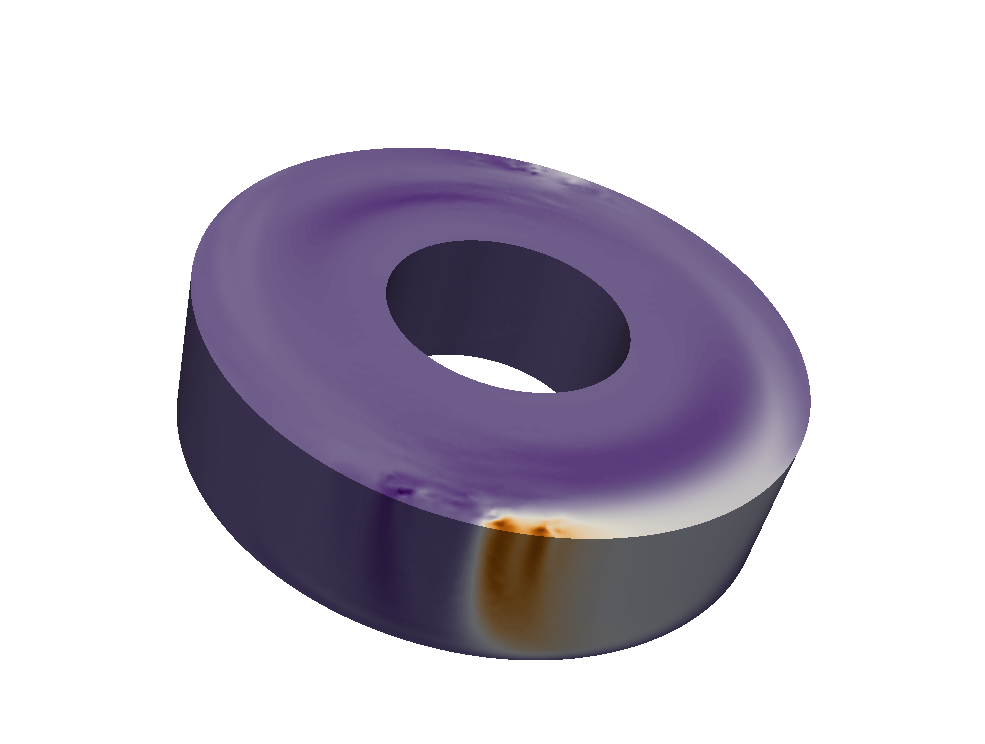}
    \end{minipage}
\begin{minipage}{.3\textwidth}
  \setlength\figureheight{5cm}
  \setlength\figurewidth{1.8cm}
  \includegraphics[]{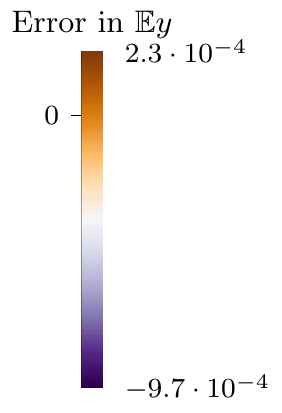}
    \end{minipage}
\caption{The difference in $\mathbb Ey$ computed via the \pcex 5 and the POD
approximation of dimension \texttt{\poddim=9} on the base of \pcex 2.}\label{fig:pcepoddiff}
\end{figure}

Motivated by \cite[Example 3.1]{GarU17}, we consider a stationary convection
diffusion problem as in \eqref{eqn:poissonproblem} with uncertainty in the
conductivity coefficient. 

As the geometrical setup, let $\Omega\subset\mathbb R^{3} $ be a cylindrical domain
of radius $R_{\textsf{o}}=1$ without its core of radius $R_{\textsf{i}}=0.4$
that is subdivided into $4$ subdomains $\Omega_i$, $i=1,2,3,4$, as illustrated in
\Cref{fig:phys-facet-regs}.

To model the uncertainty in the conductivity coefficient $\kappa$, independently on each subdomain $\Omega_i$, we assume $\kappa$ to be a random variable of a random parameter $\alpha_i$ via
\begin{equation*}
    \kappa \bigl|_{\Omega_i} = \bar \kappa + \alpha_i 
\end{equation*}
where $\bar \kappa$ is a reference value, and write $\kappa(\alpha)$ to express the dependence on the random parameter. 

In the presented example, we set $\bar \kappa = 5\cdot 10^{-4}$ and let $\alpha_i$ be
uniformly distributed on 
\begin{equation*}
\Gamma_i = [-2 \cdot 10^{-4}, 2\cdot 10^{-4}], \quad\text {for }i=1,2,3,4.
\end{equation*}

As for boundary conditions, we apply zero Dirichlet conditions at the bottom of
the domain and zero Neumann conditions elsewhere.

Without particular intentions, the convection $b$ is chosen as 
\begin{equation*}
    b(s_1, s_2, s_3) =
    \begin{bmatrix} 
        (s_1^2+s_2^2-1)s_2 \\
        -(s_1^2+s_2^2-1)s_1 \\
        s_1^2\sin(2s_3)
\end{bmatrix}
\end{equation*}
and the inhomogeneity as 
\begin{equation*}
    f(s_1, s_2, s_3) =
    \begin{cases}
        -\sin(2\pi s_1) \sin(4 \pi s_2) s_3 (0.5-s_3), & \quad\text{for }(s_1,
        s_2, s_3) \in \Omega_1 \cup \Omega_3, \\
        0, & \quad\text{for }(s_1, s_2, s_3) \in \Omega_2 \cup \Omega_4;
    \end{cases}
\end{equation*}
see \Cref{fig:sol} for a snapshot of the solution at $\kappa(0)=\bar
\kappa$.

Moreover, we use $Cy$ defined as the spatially averaged value of $y$ over a
concentric annular ring of diameter $0.1$ that is aligned with the inner
boundary at the top surface of the domain; see \Cref{fig:phys-facet-regs}
for the arrangement of the domain of observation.

The values of interest of this numerical study are the expected value $\mathbb
E$ and the variance $\mathbb V$ of $Cy$ that we approximate by a PCE with
various levels of refinement. 

For the spatial discretization, we use continuous and piecewise linear finite
elements on a discretization of the domain by tetrahedra. Although the mesh is
refined at the critical parts, namely the edges of the domain and the surfaces
where the observation is taken and the Dirichlet condition is applied, we need
about $150,000$ degrees of freedom for the spatial dimension to have a relative
error with respect to the finest considered discretization of less than
$10^{-4}$; see \Cref{tab:meshtest}.

In the experiments we used PCE with the same number of degrees of freedom
\texttt{pcedim} for all uncertainty dimensions and write \texttt{pce[d]} to
refer to the PCE discretization of dimension $d$ as well as the expected
value/variance
of $Cy$ based on this discretization. 
As can be seen in \Cref{tab:pcedimsexpy}, for computing the expected value/variance
of $Cy$, convergence of the PCE discretization is achieved already for low
dimensions. However, although the computations are well parallelized, the
computation times for the moderate PCE discretizations are already in the order
of days; see \Cref{tab:pcedimsexpy}.

\begin{table}
    \begin{tabular}{rc}
      Spatial DOFs & $Cy\bigl |_{\kappa = 4\cdot 10^{-4}} $ \\
        \hline
      \texttt{56951} & $1.116582$ \\
      \texttt{72206} & $1.077443$ \\
      \texttt{90458} & $1.069372$ \\
      \texttt{127771} & $1.069769$ \\
      \texttt{154545} & $1.065885$ \\
      \texttt{192786} & $1.065997$ \\
      \texttt{237941} & $1.064628$
    \end{tabular}
    \caption{Computed $Cy$ at $\kappa = 4\cdot 10^{-4}$ versus the number of
    degrees of freedom for the spatial discretization.}
    \label{tab:meshtest}
\end{table}

\begin{table}
    \begin{tabular}{cccc}
        \pcedim & \texttt{pce[.]} & \ctime~[s] & difference to \pcefive\\
        \hline
      \texttt{2} & $0.8809823$/$0.00897246$ & \texttt{1248.45} & $ -3.9\cdot
      10^{-5}$/ $ -1.0\cdot 10^{-4}$  \\
      \texttt{3} & $0.8810921$/$0.00908018$ & \texttt{7097.21}& $ \phantom{-}7.1\cdot
      10^{-5}  $/$ \phantom{-}9.8\cdot 10^{-6}$   \\
      \texttt{4} & $0.8810151$/$0.00907037$ & \texttt{20059.9} & $ -6.0\cdot
      10^{-6}$  /$-4.2\cdot 10^{-7}$   \\
      \texttt{5} & $0.8810211$/$0.00907079$ & \texttt{49365.4} & ---
    \end{tabular}
    \caption{The computed expected value/variance of $Cy$ based on a PCE discretization, the runtime of its computation, as well as the difference to the value of the finest computed discretization
    versus the dimension of the PCE.}
    \label{tab:pcedimsexpy}
\end{table}

This gives motivation for the use of the Galerkin POD approach that, as we will
prove, is capable to improve the estimate of a coarse PCE discretization by one
order of magnitude with little computational overhead. 

For that, we use the tensor of coefficients of the \pcetwo~discretization to
compute a basis for the space discretization that is optimal in terms of 
\Cref{thm:multidimpodspaces}. We set up the reduced models of varying size
(which we denote by \poddim) and compare the computed differences to the
expected value/variance of \pcefive~for various levels of PCE; see
\Cref{tab:poderrrandpce}. 

The distribution of the error of the POD approximation of the expected value of
the variable $y$ is plotted in \Cref{fig:pcepoddiff}.

Note that because PCE(2) leads to $2^4=16$ snapshots, POD dimensions larger
than $16$ do not add additional information to the reduced system; cp. also
\Cref{tab:projection-error} where we tabulate the projection error from
the POD reduction as defined in \eqref{eqn:projectionerror}.

We find that, for the expected value $\mathbb E y$, 
with \texttt{\poddim=6} the reduced order model recovers the
difference between \pcetwo~and \pcefive~and that for \texttt{\poddim=15}
and \texttt{\poddim=16} and a \pcedim~that exceeds the training data, the
approximation error is in the order of finer PCE
discretizations with the fine model, which is about $6\cdot 10^{-6}$; compare
\Cref{tab:pcedimsexpy} and \Cref{tab:poderrrandpce}.

As for the timings, we note that for these small POD dimensions, the effort for
computing the POD modes (around \texttt{5}s) and evaluating the reduced
models (around \texttt{0.5}s) is negligible if compared to the time to compute
the data or even the evaluation of \pcefive~with the full model; see 
\Cref{tab:pcedimsexpy}.

These results show that with the multidimensional Galerkin-POD reduction, we can use the
\pcetwo~data to compute an approximation to the expected value that is more
accurate than \pcefour~in just a $1/16$th of the computational time
(about \texttt{1253} vs. \texttt{20059.9} seconds.)

As for the approximation of the variance $\mathbb V Cy$, the Galerkin-POD
reduced model (see \autoref{tab:poderrpcebasevrnc} significantly improves the \pcex 2 approximation and almost reaches the accuracy 
of \pcex 3 in less than 1/5 of the computational time
(about \texttt{1253} vs. \texttt{7097} seconds.)

\begin{table}
    \begin{tabular}{ccccc}
      \poddim & \pcetwo & \pcethree & \pcefour & \pcefive  \\
        \hline
      \texttt{3 } &  $2.99\cdot 10^{-4}$ & $2.60\cdot 10^{-4}$ & $2.59\cdot 10^{-4}$ & $2.59\cdot 10^{-4}$ \\
      \texttt{6 } &  $3.34\cdot 10^{-5}$  & $1.05\cdot 10^{-6}$ & $1.20\cdot 10^{-6}$ & $1.20\cdot 10^{-6}$ \\
      \texttt{9 } &  $3.86\cdot 10^{-5}$ & $3.51\cdot 10^{-6}$ & $3.29\cdot 10^{-6}$ & $3.29\cdot 10^{-6}$ \\
      \texttt{12} &  $3.88\cdot 10^{-5}$ & $1.09\cdot 10^{-5}$ & $1.09\cdot 10^{-5}$ & $1.09\cdot 10^{-5}$ \\
      \texttt{15} &  $3.88\cdot 10^{-5}$ & $8.27\cdot 10^{-6}$ & $8.26\cdot 10^{-6}$ & $8.26\cdot 10^{-6}$ \\
\texttt{16} &$3.88\cdot 10^{-5}$ &  $4.48\cdot 10^{-6}$ & $4.36\cdot 10^{-6}$ & $4.37\cdot 10^{-6}$
\end{tabular}
\caption{Absolute value of the error in the POD approximation of $\mathbb E Cy$
  for various POD dimensions of the spatial discretization and various PCE
  levels. The POD approximation is based on the data of \pcetwo, i.e. $16$
  snapshots located at corresponding quadrature points.}
\label{tab:poderrrandpce}
\end{table}

\begin{table}
    \begin{tabular}{ccccc}
      \poddim & \pcetwo & \pcethree & \pcefour & \pcefive  \\
        \hline
      \texttt{3 } &$2.46\cdot 10^{-4}$ & $1.56\cdot 10^{-4}$ & $1.55\cdot 10^{-4}$ & $1.55\cdot 10^{-4}$ \\
      \texttt{6 } &$9.72\cdot 10^{-5}$ & $9.59\cdot 10^{-6}$ & $9.07\cdot 10^{-6}$ & $9.07\cdot 10^{-6}$ \\
      \texttt{9 } &$9.88\cdot 10^{-5}$ & $1.23\cdot 10^{-5}$ & $1.17\cdot 10^{-5}$ & $1.17\cdot 10^{-5}$ \\
      \texttt{12} &$9.83\cdot 10^{-5}$ & $1.50\cdot 10^{-5}$ & $1.46\cdot 10^{-5}$ & $1.46\cdot 10^{-5}$ \\
      \texttt{15} &$9.83\cdot 10^{-5}$ & $1.44\cdot 10^{-5}$ & $1.41\cdot 10^{-5}$ & $1.41\cdot 10^{-5}$ \\
      \texttt{16} &$9.83\cdot 10^{-5}$ & $1.33\cdot 10^{-5}$ & $1.29\cdot 10^{-5}$ & $1.29\cdot 10^{-5}$
\end{tabular}
\caption{Absolute value of the error in the POD approximation of the
variance $\mathbb V Cy$ for various POD dimensions of the spatial discretization
and various PCE levels. The POD approximation is based on the data of \pcetwo,
i.e. $16$ snapshots located at corresponding quadrature points.}
\label{tab:poderrpcebasevrnc}
\end{table}

To illustrate the fundamental benefit of including the PCE expansion in the POD
definition via the product space approach, we investigate the approximation by
reduced models based on random snapshots. 
It turns out that, for the same number of snapshots as with the PCE approach,
the approximation errors may reach a similar level but slightly higher level as the snapshots based on the
PCE abscissae; see \autoref{tab:poderrrandpcemcbasmd}. However, the randomness in
the snapshots makes the approximation unreliable. In fact, the median of $10$
samples gave a worse approximation than the median of $5$ samples.
In the worse case, the error level is one order of magnitude above the error that is achieved
with the same effort via the PCE based reduction $2$; see \autoref{tab:pcedimsexpy}. 

Interestingly, for the approximation of the variance $\mathbb V Cy$, the reduced
model based on random snapshots performs as well as the PCE based reduction; see
\autoref{tab:poderrrandpcemcbasvrnc}.

\begin{table}
    \begin{tabular}{ccccc}
      \poddim & \pcetwo & \pcethree & \pcefour & \pcefive  \\
        \hline
      \texttt{3 } &$1.3\cdot 10^{-4}$/$2.4\cdot 10^{-4}$ & $1.9\cdot 10^{-4}$/$2.0\cdot 10^{-4}$ & $1.9\cdot 10^{-4}$/$2.0\cdot 10^{-4}$ & $1.9\cdot 10^{-4}$/$2.0\cdot 10^{-4}$ \\
      \texttt{6 } &$2.9\cdot 10^{-4}$/$4.2\cdot 10^{-5}$ & $2.3\cdot 10^{-4}$/$9.2\cdot 10^{-5}$ & $2.2\cdot 10^{-4}$/$9.2\cdot 10^{-5}$ & $2.2\cdot 10^{-4}$/$9.2\cdot 10^{-5}$ \\
      \texttt{9 } &$8.3\cdot 10^{-5}$/$8.2\cdot 10^{-5}$ & $1.3\cdot 10^{-4}$/$1.3\cdot 10^{-4}$ & $1.3\cdot 10^{-4}$/$1.5\cdot 10^{-4}$ & $1.3\cdot 10^{-4}$/$1.5\cdot 10^{-4}$ \\
      \texttt{12} &$9.1\cdot 10^{-5}$/$2.4\cdot 10^{-5}$ & $3.3\cdot 10^{-5}$/$5.5\cdot 10^{-5}$ & $3.2\cdot 10^{-5}$/$6.2\cdot 10^{-5}$ & $3.2\cdot 10^{-5}$/$1.0\cdot 10^{-4}$ \\
      \texttt{15} &$1.2\cdot 10^{-5}$/$1.1\cdot 10^{-5}$ & $5.3\cdot 10^{-6}$/$5.0\cdot 10^{-5}$ & $4.7\cdot 10^{-6}$/$7.0\cdot 10^{-5}$ & $4.7\cdot 10^{-6}$/$7.3\cdot 10^{-5}$ \\
      \texttt{16} &$2.6\cdot 10^{-5}$/$2.5\cdot 10^{-5}$ & $2.5\cdot 10^{-5}$/$4.7\cdot 10^{-5}$ & $1.4\cdot 10^{-5}$/$4.3\cdot 10^{-5}$ & $7.3\cdot 10^{-6}$/$3.8\cdot 10^{-5}$   
\end{tabular}
\caption{Absolute value of the error in the POD approximation of $\mathbb E_\alpha Cy$ for various POD
dimensions of the spatial discretization based on $16$ random snapshots (median
value out of $5$/$10$ realizations)}
\label{tab:poderrrandpcemcbasmd}
\end{table}

\begin{table}
    \begin{tabular}{ccccc}
      \poddim & \pcetwo & \pcethree & \pcefour & \pcefive  \\
        \hline
      \texttt{3 } &$5.7\cdot 10^{-4}$/$2.9\cdot 10^{-5}$ & $6.9\cdot 10^{-4}$/$6.1\cdot 10^{-5}$ & $6.9\cdot 10^{-4}$/$6.2\cdot 10^{-5}$ & $6.9\cdot 10^{-4}$/$6.2\cdot 10^{-5}$ \\
      \texttt{6 } &$4.7\cdot 10^{-5}$/$7.7\cdot 10^{-5}$ & $1.6\cdot 10^{-4}$/$2.8\cdot 10^{-5}$ & $1.6\cdot 10^{-4}$/$2.9\cdot 10^{-5}$ & $1.6\cdot 10^{-4}$/$2.9\cdot 10^{-5}$ \\
      \texttt{9 } &$6.4\cdot 10^{-5}$/$1.0\cdot 10^{-4}$ & $3.6\cdot 10^{-5}$/$2.9\cdot 10^{-5}$ & $3.7\cdot 10^{-5}$/$2.0\cdot 10^{-5}$ & $3.7\cdot 10^{-5}$/$9.7\cdot 10^{-6}$ \\
      \texttt{12} &$9.1\cdot 10^{-5}$/$1.0\cdot 10^{-4}$ & $9.6\cdot 10^{-6}$/$5.5\cdot 10^{-6}$ & $1.0\cdot 10^{-5}$/$1.1\cdot 10^{-5}$ & $1.0\cdot 10^{-5}$/$9.8\cdot 10^{-6}$ \\
      \texttt{15} &$1.0\cdot 10^{-4}$/$1.1\cdot 10^{-4}$ & $1.3\cdot 10^{-5}$/$1.0\cdot 10^{-5}$ & $1.3\cdot 10^{-5}$/$1.0\cdot 10^{-5}$ & $1.3\cdot 10^{-5}$/$8.9\cdot 10^{-6}$ \\
      \texttt{16} &$1.0\cdot 10^{-4}$/$1.0\cdot 10^{-4}$ & $1.6\cdot 10^{-5}$/$4.7\cdot 10^{-6}$ & $4.1\cdot 10^{-5}$/$1.1\cdot 10^{-5}$ & $3.8\cdot 10^{-5}$/$1.3\cdot 10^{-5}$   
\end{tabular}
\caption{Absolute value of the error in the POD approximation of the variance $\mathbb V_\alpha Cy$ for various POD
dimensions of the spatial discretization based on $16$ random snapshots (median
value out of $5$/$10$ realizations)}
\label{tab:poderrrandpcemcbasvrnc}
\end{table}

Thus, we conclude that because of the randomness that is
not compensated by an improved performance, a POD based on random snapshots is
not well suited to approximate a system with uncertain coefficients. 

This is also indicated by the behavior of the projection error that we quantify
as follows. If $\{y(\alpha^i)\}_{i=1,\dotsc, k}$ with $a^i=(\alpha_1^i, \alphaij 2i, \alphaij 3i, \alphaij 4i)$ is a realization of a set
of snapshots, then the corresponding $k' $ POD modes are the $k' $ leading left
singular vectors of the matrix
\begin{equation*}
    \mathbf L_{\mathcal Y}^\trp 
    \begin{bmatrix}
        y(\alpha^1) & y(\alpha^2) & \dotsm y(\alpha^k) 
    \end{bmatrix},
\end{equation*}
where $\mathbf L_{\mathcal Y}$ is a Cholesky factor of the mass matrix $\My$~of
the finite element discretization; cp. \Cref{thm:multidimpodspaces}. Let
those singular vectors be the columns of the matrix $V_{\mathcal Y,k' }$. Then the
projection of the snapshots reads
\begin{equation*}
    \Ly^{-\trp}V_{\mathcal Y,k' }V_{\mathcal Y,k' }^\trp \Ly
    \begin{bmatrix}
        y(\alpha_1) & y(\alpha_2) & \dotsm y(\alpha_k) 
    \end{bmatrix}
\end{equation*}
and the projection error in the estimated mean of $Cy$ becomes
\begin{equation}\label{eqn:randsnapprojerr}
    e_{Cy;k,k' } := \frac{1}{k}\|  C [I-\Ly^{-\trp}V_{\mathcal Y,k' }V_{\mathcal
    Y,k' }^\trp \Ly] 
    \begin{bmatrix}
        y(\alpha_1) & y(\alpha_2) & \dotsm y(\alpha_k) 
    \end{bmatrix}\|_{1} .
\end{equation}

For the case of $16$ random snapshots, unlike the PCE case tabulated in Table
\ref{tab:projection-error}, the projection error stagnates at the level of
$10^{-10}$ (see \Cref{tab:randsnap-ld-projection-error}) and only drops down
to machine precision for $k=k' $, where the projection becomes the identity.
More random snapshots do not improve this situation; see the lower row of 
\Cref{tab:randsnap-ld-projection-error} where we report the projection errors for
$80$ random snapshots. In fact, the reduced models based on $80$ random
snapshots did not provide a measurable improvement over the results displayed in
\Cref{tab:poderrrandpcemcbasmd} so that we do not report them here.


All numerical computations were parallelized in $16$ threads and performed on
 a cluster computing node with $2$ Intel Xeon Silver 4110 CPUs with $2.10$GHz,
 $2\cdot 8$ virtual cores and $188$GB RAM. The reported timings are the
minimum wall time out of $5$ runs. The codes that set up, perform, and post
process the numerical examples as well as the raw data of the presented cases
are available as laid out in Figure \ref{fig:linkcodndat}.

\begin{table}
    \begin{tabular}{c|cccccc}
      \texttt{\poddim} &\texttt{3} &\texttt{6} &\texttt{9} &\texttt{12}
                       &\texttt{15} &\texttt{16} \\
                       \hline
      Projection error & $5.1\cdot 10^{-6}$&$ 3.1\cdot 10^{-8}$&$ 3\cdot
      10^{-9}$&$ 9.5\cdot 10^{-12}$&$ 3\cdot 10^{-14} $&$ 2.4\cdot 10^{-15}$
  \end{tabular}
  \caption{The projection error for varying dimension of the reduced space}
  \label{tab:projection-error}
\end{table}

\begin{table}
    \begin{tabular}{c|cccccc}
      $\texttt{k'}$ &\texttt{3} &\texttt{6} &\texttt{9} &\texttt{12}
                       &\texttt{15} &\texttt{16} \\
                       \hline
      $e_{Cy;16, k'}$  & $5.96\cdot 10^{-6}$ & $ 1.1\cdot 10^{-7}$ & $1.88\cdot 10^{-8} $&$ 3.99\cdot 10^{-9} $ &$
      2.34\cdot 10^{-10} $&$ 3.1\cdot 10^{-15}$ \\
      $e_{Cy;80,k'}$ & $5.94\cdot 10^{-6} $&$ 8.34\cdot 10^{-8} $&$ 4.03\cdot 10^{-8} $&$ 1.37\cdot 10^{-8} $ &
      $8.99\cdot 10^{-9} $ & $8.37\cdot 10^{-9}$
  \end{tabular}
  \caption{The projection error in the estimated mean as defined in
  \eqref{eqn:randsnapprojerr} for $k=16$ and $k=80$ random snapshots and for
  varying dimension \poddim~of
  the reduced space (median values out of $5$ realizations).}
  \label{tab:randsnap-ld-projection-error}
\end{table}


\section{Verification of the Approach}

The presented numerical example showed that the proposed Galerkin-POD reduction
leads to a significant speedup and memory savings in the PCE approximation. 

In order to verify the PCE approach for uncertainty quantification for
convection-diffusion problems as considered above, we present two illustrative examples that have similar
dynamics but that allow for an analytic expression of the expected values and
variances as
well as for extensive \emph{Monte Carlo} simulations for comparison.

The examples are motivated by the observation that for $b=0$ in
\eqref{eqn:poissonproblem},
the solution $y$ to the discrete problem
is given as
\begin{equation}
  y(\alpha) = \Aal ^{-1} f
\end{equation}
where $\Aal$ is the discrete Laplacian and $f$ is the right hand side. 
For this problem and a given observation operator $C$, the expected value of
$Cy$ is given as
\begin{equation*}
  \int_\Gamma C\Aal ^{-1} f \pinva{\alpha} 
\end{equation*}

For the first example, we mimick the situation that the diffusion coefficient
is constant in space and dependent on a univariate distribution so that
$\Aal=\alpha_1 \Aa$ and so that, for the solution
$\yxo(\alpha_1)=\frac{1}{\alpha_1}\Aal^{-1}f$, the expected value reads
\begin{equation*}
  \Eyxo = \int_\Gamma C\Aal ^{-1} f \pinva{\alpha} 
  =
  C\int_{\Gamma_1} \frac{1}{\alpha_1} \pinva{\alpha_1}\Aa ^{-1}  f,
\end{equation*}
which, for $\alpha_1$ being uniformly distributed on $\Gamma_1 = [\underline
\alpha_1, \overline \alpha_1]$, becomes
\begin{equation*}
  \Eyxo = \frac{1}{\overline \alpha _1 - \underline \alpha _1}\int^ {\overline \alpha _1} _{\underline \alpha _1}
  \frac{1}{\alpha} \inva \alpha C \Aa^{-1}f.
\end{equation*}
With the same arguments, the variance can be computed as 
\begin{equation*}
  \Vyxo = \frac{1}{\overline \alpha _1 - \underline \alpha _1}\int^ {\overline \alpha _1} _{\underline \alpha _1}
  \frac{1}{\alpha^2} \inva\alpha ( C \Aa^{-1}f)^2 - \Eyxo^2.
\end{equation*}

Since $C$, $\Aa^{-1}$, and $f$ are but constant factors, we can set them to $1$
and the expected
performance of PCE or \emph{Monte Carlo} for such a case can be analyzed by
their performance in the numerical integration of the integral 
\begin{equation}\label{eq:eyxo-vyxo}
\Eyxo = \frac{1}{\overline \alpha _1 - \underline \alpha _1}\int^ {\overline
\alpha _1} _{\underline \alpha _1}\frac{1}{\alpha} \inva \alpha
\quad\text{or}\quad
\Vyxo = \frac{1}{\overline \alpha _1 - \underline \alpha _1}\int^ {\overline
\alpha _1} _{\underline \alpha _1}\frac{1}{\alpha^2} \inva \alpha - \Eyxo^2.
\end{equation}

For the second example, we set
\begin{equation*}
  \alpha = (\alpha_1, \alpha_2), \quad 
  \Aal = 
  \begin{bmatrix} 
    \alpha_1 & \epsilon \\ \epsilon & \alpha_2
  \end{bmatrix}, \quad
  f = 
  \begin{bmatrix}
    1 \\ 1
  \end{bmatrix}, \quad
  C = 
  \begin{bmatrix}
    1 & 1
  \end{bmatrix}
\end{equation*}
so that, with $\alpha_1$ as above and $\alpha_2$ being distributed uniformly on
$\Gamma_2=[\underline \alpha_2, \overline \alpha_2]$, the expected value and the
variance for the corresponding solution $\yxt(\alpha_1, \alpha_2)$ read
\begin{equation}\label{eq:eyxt-vyxt}
  \begin{split}
    \Eyxt &= \frac{1}{\overline \alpha _1 - \underline \alpha_1}
  \frac{1}{\overline \alpha _2 - \underline \alpha_2}
\int^{\overline \alpha _1}_{\underline \alpha_1}
\int^{\overline \alpha _2}_{\underline \alpha_2}
\frac{1}{\alpha_1\alpha_2 - \epsilon^2}(\alpha_1 + \alpha_2 - 2\epsilon)
\inva \alpha_2 \inva \alpha_1
\quad\text{and}\\
\Vyxt &= \frac{1}{\overline \alpha _1 - \underline \alpha_1}
  \frac{1}{\overline \alpha _2 - \underline \alpha_2}
\int^{\overline \alpha _1}_{\underline \alpha_1}
\int^{\overline \alpha _2}_{\underline \alpha_2}
\biggl[
\frac{1}{\alpha_1\alpha_2 - \epsilon^2}(\alpha_1 + \alpha_2 - 2\epsilon)
\biggr]^2
\inva \alpha_2 \inva \alpha_1- (\Eyxt)^2.
\end{split}
\end{equation}

This example simulates the case of a diffusion process with two compartments
with different random diffusion parameters and with a constant $\epsilon$ as the
parameter of the coupling.

For the two examples \eqref{eq:eyxo-vyxo} and \eqref{eq:eyxt-vyxt}, we use the parameters
\begin{equation*}
  \underline \alpha_1 = \underline \alpha_2 = 3\cdot 10^{-4}, \quad
  \overline \alpha_1 =\overline \alpha_2 = 7\cdot 10^{-4}, \quad
  \epsilon = 1\cdot 10^{-4}
\end{equation*}
and compute the reference values for the means and variances
\begin{equation*}
  \Eyxo=2118.24465097,\quad
  \Vyxo=274944.360550,\quad
  \Eyxt=3504.22709343 ,\quad
  \Vyxt=261037.034256,
\end{equation*}
via evaluating the integrals with the help of a computer algebra package.

With the reference values at hand, we can estimate the approximation quality of
the PCE and MC simulations. The PCE simulation provides stable and quickly
converging approximations of the expected values and variances for the example
problems; see \Cref{tab:pce-eyxo-vyxo} and \Cref{tab:pce-eyxt-vyxt}.
Opposed to that, plain Monte Carlo simulations, show very slow convergence; see
\Cref{tab:mc-eyxo-vyxo} and \Cref{tab:mc-eyxt-vyxt}. In fact, for
example, for estimating the expected value $\Eyxt$ up to a relative error in the
order of $10^{-5}$, it takes $1,000,000$ Monte Carlo simulations or $16$
simulations for the \pcex 4 approximation. 

Since the numerical example of \Cref{sec:num-example} has a similar structure
as the two illustrative examples of this section, we conclude that the proposed
PCE discretization is well suited for this kind of multivariate uncertainty
quantification. Also, we note that a plain Monte Carlo simulation for
verification purposes is infeasible in the large-scale setup as in
\Cref{sec:num-example}, where one single forward simulation lasts about one
minute.

\begin{table}
\begin{tabular}{c|cccc}
    Method &
\pcex 3 & \pcex 4 & \pcex 5 & \pcex 6 \\
    \hline
    Relative error for $\Eyxo$ &
$-1.18\cdot 10^{-4}$ & $-5.24\cdot 10^{-6}$ & $-2.31\cdot 10^{-7}$ & $-1.01\cdot 10^{-8}$ \\
    Relative error for $\Vyxo$ &
$-1.00\cdot 10^{-2}$ & $-6.21\cdot 10^{-4}$ & $-3.51\cdot 10^{-5}$ & $-1.88\cdot 10^{-6}$
  \end{tabular}
  \caption{Approximation errors for the 1D problem \eqref{eq:eyxo-vyxo} with PCE
    discretizations \pcex N~with \texttt{N} degrees of freedom in the
  uncertainty dimension.}
  \label{tab:pce-eyxo-vyxo}
\end{table}

\begin{table}
\begin{tabular}{c|ccc}
    Method &
    \mcx {10,000} & \mcx {100,000} & \mcx{1,000,000} \\
    \hline
    Relative error for $\Eyxo$ &
    $\phantom{-} 4.01\cdot 10^{-4}$ & $-1.37\cdot 10^{-4}$ & $-9.35\cdot 10^{-5}$ \\
    Relative error for $\Vyxo$ &
    $\phantom{-} 2.36\cdot 10^{-3}$ & $\phantom{-} 1.82\cdot 10^{-3}$ & $\phantom{-} 4.21\cdot 10^{-4}$
  \end{tabular}
  \caption{Approximation errors for the 1D problem \eqref{eq:eyxo-vyxo} with Monte
  Carlo simulations \mcx N~with \texttt{N} simulations. (Median value out of 15
realizations).}
  \label{tab:mc-eyxo-vyxo}
\end{table}

\begin{table}
\begin{tabular}{c|cccc}
    Method &
\pcex 3 & \pcex 4 & \pcex 5 & \pcex 6 \\
    \hline
    Relative error for $\Eyxt$ &
$-1.23\cdot 10^{-4}$ & $-6.01\cdot 10^{-6}$ & $-2.91\cdot 10^{-7}$ & $-1.41\cdot 10^{-8}$ \\
    Relative error for $\Vyxt$ &
$-1.20\cdot 10^{-2}$ & $-8.16\cdot 10^{-4}$ & $-5.07\cdot 10^{-5}$ & $-2.98\cdot 10^{-6}$
  \end{tabular}
  \caption{Approximation errors for the 2D problem \eqref{eq:eyxt-vyxt} with PCE
    discretizations \pcex N~with \texttt{N} degrees of freedom in every
  uncertainty dimension.}
  \label{tab:pce-eyxt-vyxt}
\end{table}

\begin{table}
\begin{tabular}{c|ccc}
    Method &
    \mcx {10,000} & \mcx {100,000} & \mcx{1,000,000} \\
\hline
    Relative error for $\Eyxt$ &
    $\phantom{-}7.09\cdot 10^{-4}$ & $\phantom{-}1.11\cdot 10^{-4}$ & $-3.71\cdot 10^{-5}$ \\
    Relative error for $\Vyxt$ &
$-4.75\cdot 10^{-3}$ & $-3.61\cdot 10^{-4}$ & $-3.74\cdot 10^{-4}$
  \end{tabular}
  \caption{Approximation errors for the 2D problem \eqref{eq:eyxt-vyxt} with Monte
    Carlo simulations \mcx N~with \texttt{N} simulations. (Median value out of 11
realizations).}
  \label{tab:mc-eyxt-vyxt}
\end{table}

\newpage
\section{Conclusion}

The theory of multidimensional Galerkin POD naturally applies to problems with multivariate
uncertainties and can be made tractable for numerical experiments by exploiting
the underlying tensor structures. 
The multidimensional POD that includes Polynomial Chaos Expansions of the
candidate solutions lead to a significant efficiency gain in the uncertainty
quantification as we have illustrated in a linear convection diffusion example.
For comparison, the direct POD approach based on random snapshots is somewhat
inconclusive. In a few setups, it well competes with the PCE based
reduction but, generally, the approximation is worse and without showing reliable
trends that can be used for finding preferable configurations of number of
snapshots and dimensions of the reduced order model. 
Future work will include the investigation of POD reduction also for the PCE
dimensions and the inclusion of these reduced models for optimal control of
uncertain systems.

\begin{figure}[h]
    \begin{framed}
        \textbf{Code and Data Availability} \\
        The source code of the implementations used to compute the presented results is available from:
        \begin{center}
            \href{https://doi.org/10.5281/zenodo.4005724}{\texttt{doi:10.5281/zenodo.4005724}}\\
            \href{https://github.com/mpimd-csc/multidim-genpod-uq}{\texttt{github.com/mpimd-csc/multidim-genpod-uq}}
        \end{center}
        under the MIT license and is authored by Jan Heiland.
    \end{framed}
    \caption{Link to code and data.}\label{fig:linkcodndat}
\end{figure}


\input{gp-uq.bbl}

\end{document}

%% file: def.tex
\providecommand{\inpro}[2]{\bigl ( #1,#2 \bigr)}
\providecommand{\inprov}[2]{\bigl ( #1,#2 \bigr)_{\mathcal V}}
\providecommand{\inprovi}[2]{\int \int \dotsm \int #1 #2 \inva{}_1\inva{}_2\dotsm\inva{}_N}
\providecommand{\inva}[1]{\text{~\textup{d}} #1}
\providecommand{\pinva}[1]{\text{~\textup{d}}\mathbb P_{#1}}
\providecommand{\nspinva}[1]{\text{\textup{d}}\mathbb P_{#1}}
\def\trp{{\mathsf{T}}}  
\DeclareMathOperator{\spann}{span}
\DeclareMathOperator{\pt}{\Pi}
\def\Vec{\mathop{\mathrm {vec}}\nolimits}
\providecommand\tnsr[1]{\ensuremath{\mathbf{#1}}}
\providecommand\htnsr[1]{\ensuremath{\hat{\mathbf{#1}}}}
\providecommand{\expova}[1]{\int_{\Gamma_N} \dotsm \int_{\Gamma_2} \int_{\Gamma_1} #1 \pinva 1\pinva 2\dotsm\pinva N}

\def\ouss{Y}
\def\oust{S}
\def\ousr{W}

\def\Ouss{\ensuremath{\mathcal{\ouss}}}
\def\Oust{\ensuremath{\mathcal{\oust}}}
\def\Ousr{\ensuremath{\mathcal{\ousr}}}
\def\hOuss{\ensuremath{\hat{\mathcal{\ouss}}}}
\def\hOust{\ensuremath{\hat{\mathcal{\oust}}}}
\def\hOusr{\ensuremath{\hat{\mathcal{\ousr}}}}

\def\Oustsr{\ensuremath{{\Oust \cdot \Ouss \cdot \Ousr}}}

\def\hOustsr{\ensuremath{{\hOust \cdot \hOuss \cdot \hOusr}}}

\def\Ousv{\ensuremath{\mathcal{V}}}
\providecommand\Ousvi[1]{\ensuremath{{\mathcal{V}_{#1}}}}

\def\Hoi{H_0^1(\Omega)}

\def\Lto{L^2(\Omega)}
\providecommand\Ltgi[1]{L^2(\Gamma_{#1};\nspinva{#1})}

\def\xkotkn{\ensuremath{\mathbf x^{k_1 k_2 \dotsm k_N}}}

\def\hPsi{{\hat \Psi}}
\providecommand\psij[2]{\psi_{#1}^{#2}}
\providecommand\alphaij[2]{\alpha_{#1}^{#2}}
\providecommand\wij[2]{w_{#1}^{#2}}
\providecommand\intgi[2]{\int_{\Gamma_{#1}}#2 \pinva{#1}}

\def\My{\mathbf{M}_\Ouss}

\providecommand\Mvi[1]{\mathbf{M}_{\Ousvi {#1}}}
\providecommand\Lvi[1]{\mathbf{L}_{\Ousvi {#1}}}
\providecommand\Lvit[1]{\mathbf{L}_{\Ousvi {#1}}^\trp}

\def\Ly{\mathbf{L}_\Ouss}

\providecommand{\ppsi}[2]{\psi _{#1}^{#2}}
\providecommand{\hpsi}[2]{\hat{\psi}_{#1}^{#2}}

\providecommand{\bbmat}{\begin{bmatrix}}
\providecommand{\ebmat}{\end{bmatrix}}

\def\Aa{\ensuremath{\mathbf A}}

\def\Aal{\ensuremath{\mathbf A_{\alpha}}}
\def\hAal{\ensuremath{\hat{\mathbf A}_{\alpha}}}

\def\pcedim{\texttt{pcedim}}
\def\poddim{\texttt{k'}}
\def\ctime{\texttt{runtime}}
\def\pcefive{\texttt{pce[5]}}
\def\pcethree{\texttt{pce[3]}}
\def\pcetwo{\texttt{pce[2]}}
\def\pcefour{\texttt{pce[4]}}
\providecommand{\pcex}[1]{\texttt{pce[#1]}}
\providecommand{\mcx}[1]{\texttt{mc[#1]}}

\def\yxo{y_1}
\def\yxt{y_2}
\def\Eyxo{\mathbb E y_1}
\def\Eyxt{\mathbb E y_2}
\def\Vyxo{\mathbb V y_1}
\def\Vyxt{\mathbb V y_2}